\documentclass[11pt]{article}
\usepackage{graphicx}
\usepackage{amsfonts }
\usepackage{amsmath}
\usepackage{fullpage}
\usepackage{amssymb}
\usepackage{amsthm}
\usepackage{tikz}

\theoremstyle{definition}
\newtheorem{teo}{Theorem}[section]
\newtheorem{cor}[teo]{Corollary}
\newtheorem{lem}[teo]{Lemma}
\newtheorem{conj}[teo]{Conjecture}

\newtheorem{defi}[teo]{Definition}

\newtheorem{algo}[teo]{Algorithm}
\newtheorem{ex}[teo]{Example}

\newcommand{\B}{{\cal B}}
\newcommand{\R}{{\cal R}}
\newcommand{\I}{{\cal I}}
\newcommand{\rk}{\text{rk}}
\newcommand{\OO}{{\cal O}}
\newcommand{\FF}{{\cal F}}
\newcommand{\GG}{{\cal G}}

\DeclareMathOperator{\lk}{lk}

\begin{document}
\title{Lexicographic shellability, matroids and pure order ideals}
\author{Steven Klee\\
\small Department of Mathematics \\[-0.8ex]
\small Seattle University\\[-0.8ex]
\small Seattle, WA 98122, USA\\[-0.8ex]
\small \texttt{klees@seattleu.edu}
\and Jos\'e Alejandro Samper\\
\small Department of Mathematics\\[-0.8ex]
\small University of Washington\\[-0.8ex]
\small Seattle, WA 98195-4350, USA\\[-0.8ex]
\small \texttt{samper@math.washington.edu}
}

\maketitle
\abstract{\emph{In 1977 Stanley conjectured that the $h$-vector of a matroid independence complex is a pure $O$-sequence. In this paper we use lexicographic shellability for matroids to motivate a combinatorial strengthening of Stanley's conjecture.  This suggests that a pure $O$-sequence can be constructed from combinatorial data arising from the shelling. We then prove that our conjecture holds for matroids of rank at most four, settling the rank four case of Stanley's conjecture.  In general, we prove that if our conjecture holds for all rank $d$ matroids on at most $2d$ elements, then it holds for all matroids.  }}

\section{Introduction}
This paper studies the $h$-vector theory of matroid independence complexes. These complexes were originally motivated by trying to capture an abstract notion of linear independence in linear algebra. They have appeared in areas as diverse as graph theory, algebraic geometry, commutative algebra, and optimization. The $h$-vector also appears naturally in these contexts. Understanding $h$-vectors has driven a lot of research in the last few decades. One particular question is to understand which positive integer sequences appear as the $h$-vectors of matroids. Some conditions are known: (i) the $h$-vector of any matroid is an $O$-sequence \cite{Stanley-cm}(this is because the Stanley-Reisner ring of a matroid is Cohen Macaulay),  (ii) the $h$-vector of a matroid satisfies the Brown-Colbourn inequalities \cite{Brown-Colbourn}, and (iii) $h_i \le h_{d-i}$ for $i \le \frac d2$ \cite{Chari}. However, there are many of integer vectors that satisfy these conditions but are not $h$-vectors of matroids. The problem of completely characterizing $h$-vectors of matroids is exceedingly hard, however Stanley \cite{Stanley-cm} proposed a conjecture that narrows down the search. Recall that an integer sequence $\mathbf{a}:=(a_0, a_1, \dots, a_r)$ is a a pure $O$-sequence if there exists a pure multicomplex $\OO$ of degree $r$ that has exactly $a_i$ monomials of degree $i$ for $0\le i \le r$. 
\begin{conj}\label{Stanley}{\bf(Stanley, 1977)} The $h$-vector of an arbitrary matroid is a pure $O$-sequence.
\end{conj}
The conjecture remains open despite a tremendous amount of effort that has been put forth in trying to find a proof. Merino obtained the first positive result: he proved the conjecture for cographic matroids using the critical group of the associated graph \cite{Merino}. Schweig \cite{Schweig-LP} verified the conjecture for lattice path matroids. Oh \cite{Oh-contransversal} generalized this result to the case of cotransversal matroids by studying integer points of generalized permutohedra associated to bipartite graphs. Merino et al.  \cite{Merino-paving} proved the conjecture for paving matroids by noting that all but the last entry of the $h$-vector is determined by the dimension and the number of vertices of the matroid. H\`{a}, Stokes, and Zanello \cite{Ha-Stokes-Zanello} established the conjecture for matroids of rank three by studying properties of the level Artinian algebras. De Loera, Kemper, and Klee \cite{Deloera-Kemper-Klee} proved the conjecture combinatorially for matroids of rank 3 and corank 2 by studying the lattice of flats. They also computationally verified the conjecture for all matroids with at most $9$ elements using a database that contains all such matroids. Constatinescu and Varbaro \cite{Constantinescu-Varbaro} constructed a stratification of matroid complexes and proved the conjecture for extremal matroids in each strata. Constantinescu, Kahle and Varbaro \cite{Constantinescu-Kahle-Varbaro} proved the conjecture for proper skeleta of matroids and rank $d$ matroids with $h_d \le 5$ using again commutative algebra and level Artinian algebras. Their results that show a brute force approach to computationally disproving the conjecture is unfeasible with the current computational power. In other words,  given a matroid $h$-vector that is a candidate as a counterexample to Stanley's conjecture,  it is impossible to test all pure multicomplexes with the correct number of variables and maximal monomials against that $h$-vector in a reasonable amount of time.

Pure $O$-sequences have also been widely studied on their own. A lot of research on pure $O$-sequences has been driven by attempts to prove Stanley's conjecture. Notably, Hibi \cite{Hibi} gave a set of inequalities satisfied by pure $O$-sequences and proposed a weaker version of Stanley's conjecture. The weaker conjecture was later resolved by Chari \cite{Chari} for a larger class of PS-ear decomposable simplicial complexes. Much more can be said about pure $O$-sequences. A good reference is \cite{Boij-and-others}.

However, there seems to be a lot of skepticism about the validity of Stanley's conjecture. The classes of matroids for which the conjecture is known to hold are either too restricted or too special. Cographic and cotransversal matroids account for a very small fraction of all matroids. Paving matroids have too much structure and rank 3/corank 2 matroids do not seem to capture the full set of features and pathologies encountered in matroid theory. For instance, the simplification of a rank 3 matroid is paving, so the behavior of rank 3 matroids is similar to that of paving matroids. 

In this paper we formulate a stronger version of Conjecture \ref{Stanley} and prove that our conjecture holds for matroids of rank three and four. As matroids of rank four exhibit a less predictable behavior than those of rank three, we believe that our result provides non-trivial evidence for the validity of Stanley's conjecture as well as a new approach to proving the conjecture for all matroids. We use techniques from lexicographic shellability to get a simple decomposition of the $h$-vector that naturally gives rise to an inductive procedure to construct a pure multicomplex. In particular, instead of using the usual Tutte polynomial approach, we decompose the $h$-vector according to the independent sets disjoint from a fixed basis. We then propose a new approach in Conjecture \ref{ourconjecture} and show that it is sufficient to prove this new conjecture for a finite number of matroids of each rank. Afterwards we study matroids of rank three and four. We describe algorithms that construct a multicomplex recursively and show that the output satisfies the conditions of our conjecture, provided both algorithms produce pure $O$-sequences for the finitely many cases that have to be considered in our conjecture. All the matroids necessary for the computations have at most 8 elements and have been classified up to isomorphism in \cite{Database}, \cite{Matroid-Database} and \cite{Mayhew-Royle}. We then implemented the algorithms to computationally verify our conjecture for matroids of rank three and four in Sage \cite{sage}. We provide several examples of the output of the algorithms. \\

\noindent {\bf Acknowledgments: } We would like to thank Isabella Novik for many helpful discussions and suggestions. This project started while we were participating in the conference CoMeTA in Cortona, Italy 2013 and we would like to thank the organizers for inviting us. The second author would like to thank Federico Ardila for introducing him to the problem.
\section{Preliminaries}
A \emph{matroid} $\Delta$ is a pair $(E, \I)$ where $E$ is a finite set and $\I \subseteq2^{E}$ satisfies the following axioms. 
\begin{enumerate}
\item $\emptyset \in \I$
\item If $I \in \I$ and $I'\subseteq I$ then $I' \in \I$. (Alternatively, $\I$ is a simplicial complex.)       
\item {\bf (Extension axiom)} If $I, I' \in \I$ and $|I|<|I'|$ then there is $v\in I'$ such that $I\cup\{v\} \in \I$.
\end{enumerate}
The set $E$ is called the \emph{vertex} set of $\Delta$ and the elements of $\I$ are called the \emph{independent sets} of $\Delta$. For $A\subseteq E$ define the \emph{rank} of $A$, denoted by $\rk(A)$, to be the size of the maximum integer $k$ such that there is $I\subseteq A$ with $I\in \I$ and $|I| = k$. Abusing notation we write $\rk(\Delta) := \rk(E)$. Two matroids $\Delta=(E,\I)$, $\Delta'= (E', \I')$ are \emph{isomorphic} if there is a bijective map $f:E\to E'$ that induces a bijection from $\I$ to $\I'$.

A \emph{basis} of a matroid is an independent set that is maximal under inclusion. We denote by $\B$ the set of bases of a matroid. It follows from the extension axiom that all the bases of a matroid have the same cardinality. A subset $\B$ of $2^E$ is the collection of bases of a matroid if and only if the following conditions hold (see Chapter 1 from \cite{Oxley-book}): 
\begin{enumerate}
\item $\B \not= \emptyset$.
\item {\bf(Exchange axiom)} For $B, B' \in \B$ and $x \in B-B'$ there is $y$ in $B'-B$ such that $(B-\{x\}) \cup \{y\} \in \B$. 
\end{enumerate}
 Notice that the independent sets of the matroid can be recovered easily from the set of bases: a subset $A$ of $2^E$ is independent if and only if it is contained in some $B\in \B$. A \emph{loop} of $\Delta$ is an element $e$ of $E$ such that $\{e\} \notin \I$. A \emph{coloop} of $\Delta$ is an element $e \in E$ that is contained in every basis. We say that a matroid is a \emph{cone} if it has a coloop. 

It is sometimes useful to consider restrictions of matroids. For $A\subseteq E$ define $\Delta|_A := (A, \I|_A)$, where $\I|_A$ is the set of independent sets of $\I$ contained in $A$. It is easy to see that $\Delta|_A$ is a matroid. Denote the set of bases of $\Delta|_A$ by $\B|_A$. For a detailed introduction to the theory of matroids see \cite{Oxley-book}.

A \emph{shelling} of a matroid $\Delta$ is an ordering of the bases of $\Delta$, $B_1, \dots, \, B_k$ in such a way that for each $1\le j<i\le k$ there exists $1\le \ell < i$ and $x\in B_\ell$ such that $B_i\cap B_j \subseteq B_i\cap B_\ell = B_\ell-{x}$. Bj\"orner \cite{Bjorner-matroids} showed that every matroid admits a particularly nice shelling.  Let $\Delta$ be a matroid and order the vertices of $\Delta$ arbitrarily.  Then the lexicographic ordering of the bases of $\Delta$ with respect to this vertex order is a shelling order for $\Delta$. This shelling in fact characterizes matroids inside the wider class of shellable simplicial complexes. Any shelling of this form is called a \emph{lexicographic shelling}. 

It is a well-know fact that if $B_1, B_2, \dots, B_k$ is a shelling, then for every $B_i$ there is a subset $\R(B_i)$ of $B_i$ which is minimal with respect to not being contained in any $B_j$ with $j<i$; that is, $\R(B_i)$ is a set such that $A\subseteq B_i$ is not contained in $B_j$ for any $j<i$ if and only if $\R(B_i)\subseteq A$. These sets are called the \emph{restriction sets} of the shelling and have a very rich combinatorial structure. When there is a possibility for ambiguity, we will write $\R(B,\Delta)$ to indicate that we are considering the restriction set of $B$ as a basis in the matroid $\Delta$.  

As evidence of this combinatorial structure, define the \emph{$h$-polynomial} of a matroid by \begin{equation*} h(\Delta, x) := \sum_{i=1}^k x^{|\R(B_j)|}.\end{equation*} 
This polynomial has degree at most $d:=\rk(\Delta)$. Let $h_i$ be the coefficient of $x^i$ in $h(\Delta, x)$. The vector $h(\Delta) := (h_0, h_1, \dots, h_d)$ is called the \emph{$h$-vector} of $\Delta$ and is a very important invariant of the matroid. Consider the polynomial \begin{equation*}f(\Delta,x) = \sum_{j=0}^d f_{j}x^j := (1+x)h\left(\Delta, \frac{x}{1+x}\right).\end{equation*}
The coefficient $f_{j}$ of $x^j$ is equal to the number of independent sets of rank $j$ in $\Delta$ (see \cite{Bjorner-matroids}). In particular, this implies that the $h$-polynomial is independent of the shelling order. One major problem in the theory of matroids is to understand the possible values that the $h$-vector can take. For more details about shellability and matroids see \cite{Bjorner-matroids}.

An \emph{order ideal} or \emph{multicomplex} $\OO$ is a finite non empty collection of monomials in a finite set of variables such that if $m\in \OO$ and $m'|m$ then $m'\in \OO$. A monomial order $\OO$ is called \emph{pure} if all its maximal monomials have the same degree. For an order ideal $\OO$ let $F_i = F_i(\OO)$ denote the number of monomials of degree $i$. Assume that $d$ is the maximum degree of a monomial in $\OO$. The vector $F(\OO):=(F_0, F_1, \dots, F_d)$ is called the \emph{$F$-vector} of $\OO$. An \emph{$O$-sequence} is an integer sequence $\mathbf{a} = (a_0, a_1, \dots , a_d)$ such that there exists an order ideal $\OO$ with $F(\OO) = \mathbf{a}$. An $O$-sequence is \emph{pure} if it comes from a pure order ideal. 

Stanley \cite{Stanley-cm} showed that the $h$-vector of any matroid (and more generally any Cohen-Macaulay simplicial complex) is an $O$-sequence and posited Conjecture \ref{Stanley}. For a detailed explanation of the relationship between matroids, simplicial complexes, and commutative algebra see \cite{Stanley-greenbook}.

\section{Restriction sets of lexicographic shellings}
For a positive integer $n$ let $[n]$ denote the set of integers $\{1,2,\dots n\}$. Let $\Delta=([n], \I)$ be a matroid that has $[d]$ as a basis. There is no loss of generality in assuming that $[d]$ is a basis in $\Delta$ since we can reorder the ground set of $\Delta$ without changing its combinatorial structure.  The lexicographic order on the bases of  $\Delta$ gives a shelling. From now on we assume that $\Delta$ is endowed with this shelling. It is clear that $[d]$ is the first basis of this shelling. For bases $B, B'$ of $\Delta$ we write $B<B'$ if $B$ is smaller than $B'$ in the lexicographic order (lex order, for short) induced by the natural order of $[n]$.  Our goal now is to understand the set $\{\R(B)\ |\ B \in \mathcal{B}(\Delta)\}$. 
\begin{lem} Let $I$ be an independent set of $\Delta$ such that $I\cap [d] = \emptyset$. Then there exists a basis $B$ of $\Delta$ with the property that $\R(B) = I$.  
\end{lem}
\begin{proof} Let $B$ be the lexicographically smallest basis of $\Delta$ that contains $I$. We claim that $\R(B) = I$. First note that $\R(B) \subseteq I$ because $B$ is the first basis of the shelling order that contains $I$. On the other hand, if $v$ is an element of $I - R(B)$, then we can apply the extension axiom to $B-\{v\}$ and $[d]$ to obtain another basis $B'$ . Then $\R(B) \subseteq B'$ and $B'$ is smaller than $B$ in the lex order, which is a contradiction. 
\end{proof}

\begin{lem} Let $B$ be a basis of $\Delta$ and let $I = B-[d]$. Then $I\subseteq \R(B)$. 
\end{lem}
\begin{proof} Assume to the contrary that there is an element $v\in I- \R(B)$. Then use the extension axiom with  $B-\{v\}$ and  $[d]$ to find a basis $B' < B$ that contains $\R(B)$. This is a contradiction.
\end{proof}

Now we introduce two matroids that can be associated to a given independent set $I \in \Delta$ with $I \cap [d] = \emptyset$.  

\begin{defi}Let  $I$ be an independent set of $\Delta$ such that $I\cap [d] = \emptyset$. Let $\Gamma_I$ be the matroid  whose independent sets are subsets $G$ of $[d]$ such that $G\cup I$ is an independent of $\Delta$. Let $\B_I$ be the set of bases of $\Gamma_I$. To simplify notation we write $\B_x := \B_{\{x\}}$. Furthermore, let $\Delta_I := \Delta|_{[d]\cup I}$. 
\end{defi}

We note that $\Gamma_I$ and $\Delta_I$ are indeed matroids since $\Gamma_I = \lk_{\Delta}(I)|_{[d]}$, and links and restrictions of matroid independence complexes are also matroid independence complexes. 
\begin{teo}\label{superthm} Let $I$ be an independent set of $\Delta$ and let \[U_I := \{\R(B,\Delta)-I \, : \, B\in \B, B-[d] = I\}.\] Let $\B_I = \{G_1, \dots G_\ell\}$ be the set of bases of $\Gamma_I$ ordered lexicographically with respect to the natural order on $[d]$, and let $V_I = \{\R(G_i,\Gamma_I) \, | \, 1\le i \le \ell\}$. Then $U_I = V_I$. 
\end{teo}
\begin{proof} Let $1\le j \le \ell$. Note that $G_j \cup I$ is a basis of $\Delta$. We claim that $\R(G_j \cup I, \Delta) = \R(G_j, \Gamma_I) \cup I$. First we show that $\R(G_j\cup I,\Delta)-I \subseteq \R(G_j, \Gamma_I)$. If $\R(G_j,\Gamma_I)\cup I \subseteq F_\ell < G_j\cup I$ then $F_\ell-[d] \supseteq I$. We can assume that we have equality. Indeed, if $F_\ell -[d]\not=I$ then we can remove an element of $F_\ell-(I\cup[d])$ and extend the remaining set to a basis using elements of $[d]$. This gives a new basis that is lexicographically smaller than $F_k$ and still contains $\R(G_j,\Gamma_I)\cup I$. It follows that $F_\ell - I = G_s < G_j$ for some $s$ and then $\R(G_j, \Gamma_i)\subseteq G_s$, which is impossible by definition of $\R(G_j,\Gamma_I)$. Hence there is no such $k$ and so $\R(G_j\cup I, \Delta) \subseteq \R(G_j, \Gamma_I)\cup I$. Now let us assume that the containment is strict, i.e, there is $g \in \R(G_j,\Gamma_I)\cup I$ that is not in $\R(G_j\cup I, \Delta)$. This element is in $[d]$, because removing $[d]$ from both sides yields $I$. Then $\R(G_j\cup I,\Delta)-I \subseteq G_k < G_j$. But then $\R(G_j \cup I,\Delta) \subseteq G_k\cup I$ which is impossible. The desired equality follows. It implies that $V_I \subseteq U_I$. Since both $U_I$ and $V_I$ are finite and  $|U_I| = |V_I|$, the sets $U_I$ and $V_I$ are equal as desired. 
\end{proof}

\begin{cor}\label{decomp} Keeping the same notation as in Theorem \ref{superthm} we have: \begin{equation}h(\Delta,x) = \sum_{I \in \Delta|_{[n]-[d]}} x^{|I|}h(\Gamma_I,x).\end{equation}
\end{cor}

We illustrate Theorem \ref{goingUp} in the following example. 

\begin{ex} \label{firstfano}
 Consider the Fano matroid with its ground set labeled as in the following illustration. 

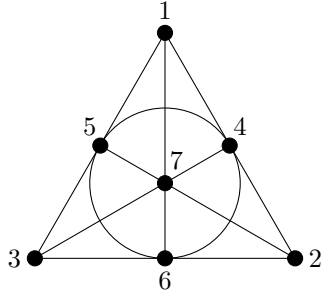
\begin{center}
\begin{tikzpicture}
\draw (0,0) circle (1);
\draw (-1.73,-1) -- (0,2) -- (1.73,-1) -- (-1.73,-1);
\draw (-1.73,-1) -- (.86,.5);
\draw (0,2) -- (0,-1);
\draw (1.73,-1) -- (-.86,.5);
\filldraw[black] (0,0) circle (.1);
\filldraw[black] (-1.73,-1) circle (.1);
\filldraw[black] (1.73,-1) circle (.1);
\filldraw[black] (0,2) circle (.1);
\filldraw[black] (.86,.5) circle (.1);
\filldraw[black] (-.86,.5) circle (.1);
\filldraw[black] (0,-1) circle (.1);
\draw (0,2.3) node {$1$};
\draw (2,-1) node {$2$};
\draw (-2,-1) node {$3$};
\draw (1,.75) node {$4$};
\draw (-1,.75) node {$5$};
\draw (0,-1.3) node {$6$};
\draw (.15,.3) node {$7$};
\end{tikzpicture}
\end{center}

The independent sets of the Fano matroid correspond to sets of at most three points that do not lie on a line or the circle.  

Let $I = \{6\}$.  In this case, $\Gamma_{I}$ has facets $\{\{1,2\},\{1,3\}\}$.  Under the lexicographic shelling on $\Gamma_I$, the corresponding restriction sets are $V_I=\{\emptyset, \{3\}\}$.  Similarly, the facets of $\Delta_I$ are $\{\{1,2,6\},\{1,3,6\}\}$.  To form the set $U_I$, we must compute the restriction set for each of these facets relatively to the entire shelling order on $\Delta$.  Since $\{1,2,6\}$ is the lexicographically smallest facet containing $I = \{6\}$, we have $\R(\{1,2,6\}) = \{6\}$.  Since $\{1,2\},\{1,6\} \subseteq \{1,2,6\}$, and $\{1,3,6\}$ is the lexicographically smallest face containing $\{3,6\}$ we have $\R(\{1,3,6\}) = \{3,6\}$.  Since $U_I$ is obtained by removing $I = \{6\}$ from each of these restriction sets, it follows also that $U_I = \{\emptyset, \{3\}\}$.  
\end{ex}
The following lemma is proved in \cite[Theorem III.3.4]{Stanley-greenbook} , but we include a proof for the sake of completeness. 
\begin{lem} Let $\Delta$ be a rank $d$ matroid. Then $h_d(\Delta) = 0$ if and only if $\Delta$ is a cone.
\end{lem}
\begin{proof} Let $B_1, B_2, \dots, B_k$ be a shelling order and consider $\R(B_k)$. Since $h_d = 0$, there is $j\in B_k - \R(B_k)$. It follows that $\R(B_k) \subseteq (B_k-j)$. Thus the only basis of $\Delta$ that contains $B_k - \{j\}$ is $B_k$. On the other hand, the exchange axiom tells us that every basis $B$ has an element $x$ such that $(B_k-j)\cup\{x\}$ is a basis. This basis contains $B_k-j$, and so it has to be $B_k$. It follows that  $x=j$, and so $\Delta$ is a cone. Now assume that $\Delta$ is a cone and let $u$ belong to every basis, then $u$ is not contained in the restriction set of any basis, so no restriction set has size $d$ and $h_d =0$. 
\end{proof}
\begin{lem}\label{goingUp} Let $\Delta$ be a rank $d$ matroid with $h_d(\Delta) \not =  0$. If $I$ is and independent disjoint from $[d]$ such that $h_{d-|I|}(\Gamma_I) = 0$, then there exists $z \in [n]-([d]\cup I)$ such that $I\cup z \in \Delta$. 
\end{lem}
\begin{proof} Since $h_d(\Delta) \not = 0$, $\Delta$ is not a cone. As $h_{d-|I|}(\Gamma_I) = 0$ we obtain that $|I|<d$ and that $\Gamma_I$ is a cone. Hence there is $u\in [d]$ that belongs to every basis of $\Gamma_I$. Since $\Delta$ is not a cone there is $z \in [n]-[d]$ such that $B:=([d]-\{u\})\cup\{z\}$ is a basis of $\Delta$. Apply the extension axiom with $I$ and $B$ to get a basis $B'$. Then $z \in B'$. If not we would be adding only elements of $[d]-\{u\}$, but there is no basis of $\Gamma_I$ that does not contain $u$. The result follows. 
\end{proof}

We are now in a position to state our new approach to Stanley's conjecture.   

\begin{defi} A \emph{based matroid} is a pair $(\Delta, B, <)$ where $\Delta$ is a matroid, $B$ is a basis of $\Delta$ and $<$ is a total order of $E(\Delta)- B$. For an independent set $I$, such that $I\cap B = \emptyset$, let $\Gamma_I$ be the matroid whose elements are subsets $U$ of $B$ with $U\cup I \in \Delta$. Two based matroids $(\Delta, B,<)$, $(\Delta', B',<')$ are \emph{isomorphic} if there is a matroid isomorphism $f: \Delta \to \Delta'$ such that $f(B) = B'$ and $f$ is order preserving on $E(\Delta)-B$. 
\end{defi}
\begin{conj}\label{ourconjecture} Let $d>1$ be a fixed integer and let ${\cal A}^d$ be the family of based matroids of rank $d$. There exists a map $\FF$ from ${\cal A}^d$ to the family of pure order ideals such that the following conditions hold  for every based matroid $(\Delta,B,<)$.
\begin{enumerate}
\item The variables of $\FF(\Delta,B,<)$ are $\{x_i \, |\, i\in E(\Delta)-B\}$. 
\item Every monomial in $\FF(\Delta,B,<)$ is supported on a set of the form $\{x_i \, | \, i\in I\}$ for some independent set $I$ of $\Delta$ with $I \cap B = \emptyset$. 
\item For each independent set $I$ that is disjoint from $B$, there are exactly $h_j(\Gamma_I)$ monomials  in $\FF(\Delta, B,<)$ with degree $|I|+j$ and support  $\{x_i \, | \, i \in I\}$. 
\item For each independent set $I$ that is disjoint from $B$, the restriction of $\FF(\Delta,B,<)$ to the variables $\{x_i \, | \, i\in I\}$ is $\FF(\Delta|_{B\cup I},B, <)$. 
\item If $(\Delta', B', <')$ is a based matroid and $f:(\Delta, B,<) \to (\Delta, B,<')$ is an isomorphism, then $\FF(\Delta, B,<)$ is naturally isomorphic to $\FF(\Delta', B',<')$ by relabeling the index of each variable in $\FF(\Delta, B,<)$ with its image under $f$. 
\end{enumerate}
\end{conj}

Most importantly, notice that Conjecture $\ref{ourconjecture}$ together with Corollary \ref{decomp} implies Stanley's conjecture (Conjecture \ref{Stanley}). To see this, note that for an arbitrary matroid $\Delta$ we can pick a basis $B$ and any order $<$ on $E(\Delta)-B$ and apply the conjecture to the based matroid $(\Delta,B,<)$.

Next, we make three remarks in defense of this conjecture as a reasonable approach to proving Stanley's conjecture.

First,  for any independent set $I$ with $I \cap [d] = \emptyset$, notice that  $\rk(\Gamma_I) = d-|I|$ and $|E(\Gamma_I)| \le d$. By the Upper Bound Theorem \cite{Stanley-UBC}, $h_j(\Gamma_I)$ is bounded by the number of monomials of degree $j$ in $|E(\Gamma_I)|-\rk(\Gamma_I) \le |I|$ variables. This shows that condition 3 in the above conjecture cannot fail on account of $h_j(\Gamma_I)$ exceeding the number of monomials of degree $|I|+j$ supported on $x^I$. 

Second, even if $\Delta$ is not a cone, one should expect that $\Gamma_I$ will be a cone for many of the independent sets $I \in \Delta$.  Therefore, if the conjecture is true, one should not expect that each of the monomials supported on $x^I$ will all divide into a monomial of degree $d$ that is also supported on $x^I$.  However, Lemma \ref{goingUp} indicates that in this case, we can expect each of the maximal monomials supported on $x^I$ to divide into a monomial of higher degree that is supported on $x^{I'}$ for some $I' \supset I$.  

Finally, the order condition may seem strange at first. However, in the case of rank 3 and 4 matroids, it is used largely as a ``tie-breaker" in the algorithms we define to construct pure $O$-sequences. Specifically, we need rules that allow us to distinguish independent sets $I, I'$ with $\Gamma_I = \Gamma_I'$. Based on evidence of the cases $d=3,4$, we believe the order may only be needed to distinguish independent sets $\{x\}, \, \{y\}$ with $h(\Gamma_{\{x\}}) = h(\Gamma_{\{y\}})$.  Every possible ordering of the ground set of $\Delta$ has to be considered in order for condition $4$ to hold, as in such restrictions we can get isomorphic matroids with same initial bases but a different underlying order. 

We now prove a theorem that will be crucial for treating the cases of rank three and rank four matroids.  

\begin{teo}\label{finite} Let ${\cal U}_{d}$ be the subset of ${\cal A}^d$ consisting of matroids with at most $2d$ vertices. Assume that there is a map $\GG$ from ${\cal U}_{d}$ to the family of pure order ideals such that $\GG(\Delta,B,<)$ satisfies the conditions on $\FF(\Delta,B,<)$ of Conjecture \ref{ourconjecture} for all $\Delta \in {\cal U}_{d}$. Then there exists a function $\FF$ that satisfies the conditions of Conjecture \ref{ourconjecture} and such that $\FF|_{{\cal U}_{d}} = \GG$.
\end{teo}

\begin{proof} Let $(\Delta,B,<)$ be a matroid of ${\cal A}^d$. For each independent set $I$ with $B\cap I = \emptyset$, the based matroid $(\Delta|_{B\cup I}, B,<)$ is in ${\cal U}_d$, so $\GG(\Delta|_{B\cup I}, B,<)$ is well-defined. Consider the set of monomials: \[\FF(\Delta,B,<) := \bigcup_{\substack{I\in \Delta, \\  I\cap B =\emptyset}} \GG(\Delta|_{B\cup I}, B,<).\]
We claim that $\FF$ is the desired map. By definition of $\FF$, the variables of $\FF(\Delta,B,<)$ are $\{x_i \, | \, i \in E(\Delta)-B\}$ so condition 1 is satisfied. As every monomial comes from some $\GG(\Delta|_{B\cup I},B,<)$, its support is a subset of $\{x_i \, | \, i\in I\}$ and hence condition 2. holds. Notice that if $I' \subset I$ then $\GG(\Delta|_{B\cup I'},B,<) \subseteq \GG(\Delta|_{B\cup I},B)$ since the restriction of $\GG(\Delta|_{B\cup I},B)$ to the variables $\{x_i \,| \, i \in I'\}$ is $\GG(\Delta|_{B\cup I' },B,<)$ by assumption. From this we obtain that the monomials in $\FF(\Delta, B,<)$ whose support is $\{x_i \, | \, i\in I\}$ form the set of such monomials in $\GG(\Delta|_{I\cup B}, B,<)$. Since $\Gamma_I(\Delta) = \Gamma_I(\Delta|_{I\cup B})$, we conclude that there are $h_j(\Gamma_I)$ monomials of degree $|I| + d$ in $\FF(\Delta, B,<)$ whose support is $\{x_i \, |\, i \in I\}$, and so condition 3 is satisfied. Conditions 4 and 5 are immediate from the definition. Therefore $\FF(\Delta)$ satisfies conditions 1 to 5. 

We now show that $\FF(\Delta, B,<)$ is a pure order ideal. Let $m$ be a monomial in \\$\FF(\Delta,B,<)$. There is an independent set $I$ such that $m\in \GG(\Delta|_{B\cup I}, B,<)$. Since $\GG(\Delta|_{B\cup I}, B,<)$ is a multicomplex, all the divisors of $m$ are in $\FF(\Delta, B,<)$, and so $\FF(\Delta, B,<)$ is a multicomplex. Let $k = \max\{i \, | \, h_i(\Delta) > 0\}$. We claim that there is $I'$ such that $I\subset I'$ and $h_k(\Delta|_{B\cup I'}) > 0$. Removing coloops (all are contained in B), we can assume that $h_d(\Delta)>0$. Then by Lemma \ref{goingUp} there is $I'$ that contains $I$ such that $h_{d-|I'|}(\Gamma_{I'}) >0$. Corollary \ref{decomp} implies  that $h_d(\Delta|_{B\cup I'}) \ge h_{d-|I'|}(\Gamma_{I'}) > 0$ as desired. Now $m$ is an element of a pure multicomplex $\GG(\Delta|_{B\cup I'}, B,<) \supset \GG(\Delta_{B\cup I}, B,<)$ of degree $k$, so it divides a monomial of degree $k$ supported in a subset of $I'$. It follows that $\FF(\Delta, B,<)$ is pure, as claimed. 
\end{proof}

Notice that ${\cal U}_d$ contains only finitely many isomorphism classes of based matroids. Therefore, to verify our conjecture for matroids of a fixed rank we only have to construct the desired order ideal for finitely many matroids. In the following sections we construct algorithms that receive as their input matroids of rank 3 and 4 respectively, and output the desired order ideal. This can be used to explicitly construct $\GG$ for $d=3, 4$, which in turn implies Conjecture \ref{ourconjecture}.

\section{Stanley's conjecture for rank 3 matroids}
We begin by verifying Conjecture \ref{ourconjecture} for rank three matroids. For this, we develop an algorithm that constructs a pure order ideal for every matroid whose vertex set is ordered. We verify with a computer's aid that the algorithm produces a suitable $\FF(\Delta, B,<)$ for all based matroids of rank $3$ with at most $6$ vertices and conclude from Theorem \ref{finite} that the conjecture holds for $d=3$. 

First we will give an overview to show how our algorithm will work.  Let $\Delta$ be a matroid of rank $d$ on ground set $[n]$ such that $[d]$ is a basis of $\Delta$.  Our goal is to construct a pure multicomplex $\mathcal{O}$ with the property that for each independent set $I \subseteq [n] -[d]$, $\mathcal{O}$ contains $h_j(\Gamma_I)$ monomials of degree $|I|+j$ supported on $x^I$. When $|I| = 1$, it is easy to verify that $h_j(\Gamma_I)$ is equal to either zero or one and that $|\B_I|$ counts the number of indices $j$ for which $h_j(\Gamma_I) = 1$.  Thus if $I = \{x\}$, we add $\{x^t\ |\ 1 \leq t \leq |\B_x|\}$ to $\OO$.

Next, we proceed to independent sets of the form $I = \{x,y\}$.  In rank 3, the $h$-vector of $\Gamma_I$ can only be $(1,0)$, $(1,1)$, or $(1,2)$ because $\Gamma_I$ can consist of either one, two, or all three of the vertices among $\{1,2,3\}$.  In each case, $h_0(\Gamma_I) = 1$, so we add $xy$ to $\mathcal{O}$.  When $h_1(\Gamma_I) =1$, we have to make a choice of whether to add $x^2y$ or $xy^2$ to $\mathcal{O}$.  This choice depends on $\B_x$ and $\B_y$.  Specifically, if $|\B_x|<|\B_y|$, then we choose to add $xy^2$ to $\mathcal{O}$.  For example, if $|\B_x|=1$ and $|\B_y|=2$, then $x^2 \notin \mathcal{O}$ so it does not make sense to add $x^2y$ to $\OO$; but $y^2 \in \OO$, and hence all divisors of $xy^2$ also belong to $\OO$.  Moreover, $xy^2$ serves as a maximal monomial that is divisible by both $x$ and $y^2$.

As a further remark, one could wonder what would happen if $I = \{x,y\}$, $h(\Gamma_I) = (1,1)$, and $|\B_x|=|\B_y|=2$.  This would indeed be problematic as we would have $x^2$ and $y^2$ in $\OO$, but neither $x^3$ nor $y^3$ in $\OO$.  Since $h(\Gamma_I) = (1,1)$, we would only be allowed to add $x^2y$ or $xy^2$ to $\OO$, but not both.  Thus only one of the monomials $x^2$ or $y^2$ would divide a maximal monomial supported on $xy$, and we might not be able to guarantee that the other would ever divide into a maximal monomial.  The following lemma will forbid such pathologies.  When $h(\Gamma_I) = (1,1)$, if $|\B_x| = 2$, then $|\B_y| = 3$.  This means that $x^2$ belongs to $\OO$, $x^3$ does not belong to $\OO$, and $y^3$ belongs to $\OO$.  Thus it is natural to add $x^2y$ to $\OO$ to maintain purity throughout the construction.  

The following lemma shows that when $I = \{x,y\}$ as above, the relationship between $h(\Gamma_I)$, $\B_x$, and $\B_y$ is not arbitrary.  The lemma can be proved directly through repeatedly applying the exchange axiom to small matroids of rank 3, but we have also verified it directly in Sage using the databases of small matroids from \cite{Matroid-Database}.  

\begin{lem} \label{rank3-link-structure}
Let $\Delta$ be a matroid of rank 3 on the ground set $[n]$ for which $[3]$ is a basis.  Let $I = \{x,y\} \subseteq [n]-[3]$ be an independent set.  Then
\begin{enumerate}
\item If $h(\Gamma_I) = (1,0)$, then there exists $z \in [n]- (\{x,y\} \cup [3])$ such that $\{x,y,z\} \in \Delta$. 
\item If $h(\Gamma_I) = (1,1)$, then one of the following holds: 
\begin{enumerate}
\item $|\B_x|=1$ and $|\B_y| \geq 2$,
\item $|\B_y|=1$ and $|\B_x| \geq 2$,
\item $|\B_x| =2$ and $|\B_y| = 3$,
\item $|\B_x| = 3$ and $|\B_y| \geq 2$.
\end{enumerate}
\item If $h(\Gamma_I) = (1,2)$, then $|\B_x| \geq 2$ and $|\B_y| \geq 2$. 
\end{enumerate}
\end{lem}

In light of these motivating examples, we present the algorithm for constructing a natural pure $O$-sequence that can be associated to any rank 3 matroid.

\begin{algo} \label{construct-multicomplex} {\bf Constructing a pure degree 3 order ideal.}\\
\boxed{{\bf INPUT:}\text{ A rank 3 matroid $\Delta$ whose vertex set is ordered}}\\
\boxed{{\bf OUTPUT: } \text{A pure order ideal $\OO$ whose }F\text{-vector is } h(\Delta).}\\
{\bf OUTLINE:}\\
{\bf STEP 0:} Reorder the vertices of the matroid as follows. Pick the lexicographic smallest basis of $\Delta$ and relabel its elements as $\{1,2,3\}$. For the remaining vertices declare $x<y$ if $|\B_x|<|\B_y|$ or $B_x <_{\text{lex}} B_y$ and preserve the relative order of the vertices for which $|B_x|$ is fixed. \\
{\bf STEP 1:} Construct the family of independent sets $I$ with $I\cap \{1,2,3\} = \emptyset$ and partition them into four collections $A_0,\, A_1, \, A_2, \, A_3$, where $A_i$ contains all the elements of size $i$. Initialize a list of monomials $\OO$ to the empty list.\\
{\bf STEP 2} Add the monomial $1$ to $\OO$. It corresponds to the empty set, the only element of $A_0$. \\
{\bf STEP 3} For each $I= \{x\}\in A_1$, add all the monomial $x^t$ to $\OO$, where $1\le t \le |B_x|$. \\
{\bf STEP 4} For $I = \{x,y\}\in A_2$ with $x<y$ we split into cases according to the values of $h(\Gamma_{I})$, $|\B_x|$, and $|\B_y|$:
\begin{enumerate}
\item If $h(\Gamma_{I}) = (1,0)$, then add the monomial $xy$ to $\OO$. 
\item If $h(\Gamma_{I}) = (1,1)$ we again split into cases: 
\begin{enumerate}
\item[i.] If $|\B_x|=1$ then add the monomial $xy^2$ to $\OO$. ($x^2$ is not on the list, so we cannot add anything that is not divisible by $x^2$).
\item[ii.] Else add the monomial $x^2y$ to $\OO$. 
\end{enumerate}
\item If $h(\Gamma_{I}) = (1,2)$, then add the monomials $xy$, $xy^2$, and $x^2y$ to $\OO$. 
\end{enumerate}
{\bf STEP 5:} For each $I = \{x,y,z\}\in A_3$ add the monomial $xyz$ to $\OO$.  
\end{algo}

For a rank 3 based matroid $(\Delta, B,<)$ let $\FF(\Delta, B,<)$ be the output of the algorithm when we relabel the vertices of $\Delta$ to have $B$ as the smallest lexicographic basis and keep the relative order on $E(\Delta)-B$. By construction of the algorithm, conditions 1 to 5 from Conjecture \ref{ourconjecture} are satisfied, and so if the output is a pure order ideal for each element of ${\cal U}_3$, then the proof of Conjecture \ref{ourconjecture} in the case of $d=3$ follows by Theorem \ref{finite}.


We have computationally verified that the output of this code produces a pure order ideal for each element of $\mathcal{U}_3$.  A summary of our code is included in Section \ref{code-section}, and all of our code is available at \cite{Code}.  Thus we have computationally verified that Stanley's Conjecture holds for matroids of rank three.

\begin{teo} Conjecture \ref{ourconjecture} holds for $d=3$. 
\end{teo}

The algorithm does more than giving us a tool to check that the hypothesis for Theorem \ref{finite} indeed hold for $d=3$, it explicitly constructs a valid map $\FF$. Furthermore, if the vertices of the matroid are ordered it gives the pure order ideal even if the number of vertices is larger than $6$. To illustrate the results of the algorithm we now present an example. 

\begin{ex}\label{fano} Once again, we consider the Fano matroid of Example \ref{firstfano}.  Under the natural ordering of the vertex set $\{1,2,3,4,5,6,7\}$, the Fano matroid has the following restricted $h$-vectors, which contribute the shown monomials to the corresponding monomial family.  It can be easily verified that the resulting family of monomials forms a pure multicomplex.
\begin{center}
\begin{tabular}{|c|c|c|} \hline
$I$ & $h(\Gamma_I)$ & Monomials \\ \hline
$\emptyset$ & $(1,0,0,0)$ & 1 \\ \hline
$\{4\}$ & $(1,1,0)$ & $x_4,\  x_4^2$  \\ \hline 
$\{5\}$ & $(1,1,0)$ & $x_5,\ x_5^2$  \\ \hline
$\{6\}$ & $(1,1,0)$ & $x_6,\ x_6^2$  \\ \hline
$\{7\}$ & $(1,1,1)$ & $x_7,\ x_7^2,\ x_7^3$  \\ \hline
$\{4,5\}$ & $(1,2)$ & $x_4x_5,\ x_4x_5^2,\ x_4^2x_5$ \\ \hline
$\{4,6\}$ & $(1,2)$  & $x_4x_6,\ x_4x_6^2,\ x_4^2x_6$\\ \hline
$\{4,7\}$ & $(1,1)$ & $x_4x_7,\ x_4^2x_7$ \\ \hline
$\{5,6\}$ & $(1,2)$ & $x_5x_6,\ x_5x_6^2,\ x_5^2x_6$ \\ \hline
$\{5,7\}$ & $(1,1)$  & $x_5x_7,\ x_5^2x_7$ \\ \hline
$\{6,7\}$ & $(1,1)$  & $x_6x_7,\ x_6^2x_7$\\ \hline
$\{4,5,7\}$ & $(1)$ & $x_4x_5x_7$ \\ \hline
$\{4,6,7\}$ & $(1)$ & $x_4x_6x_7$ \\ \hline
$\{5,6,7\}$ & $(1)$ & $x_5x_6x_7$  \\ \hline
\end{tabular}
\end{center}
%
%
\end{ex}
\section{Stanley's conjecture for rank 4 matroids}
We now proceed to prove Conjecture \ref{ourconjecture} for $d=4$. As for the rank $3$ case, we give an algorithm that explicitly constructs a pure order ideal for every matroid of rank 4 whose vertex set is ordered. We verify that the output of the algorithm satisfies the conditions of Theorem \ref{finite} for all based matroids in ${\cal U}_4$, therefore proving Conjecture \ref{ourconjecture} for rank 4 matroids and also, as a result, Stanley's conjecture \ref{Stanley}. We first study some properties of rank four matroids. 

\subsection{Structural properties of rank 4 matroids}
In order to motivate and explain why the algorithm in the next section is built as it is, we will now present a collection of structural results about the restricted $h$-vectors of matroid complexes. These lemmas admit theoretical proofs, but since it is enough to check them for matroids of rank 4 with seven elements and there are 374 such matroids, we were able to computationally verify the results in just a few seconds using Sage \cite{Code}. 

Throughout this section, $\Delta$ is a rank 4 matroid such that $\{1,2,3,4\}$ is a basis of $\Delta$. All the lemmas are local, that is, all are concerned with an independent set $I$ disjoint from $\{1,2,3,4\}$ and the properties of the studied structure only depend on the matroid $\Delta_I$ that results from restricting the ground set of $\Delta$ to $\{1,2,3,4\}\cup I$.  Further, the independent set $I$ is not a basis, hence it suffices to check the properties for matroids with at most 7 elements. The code to do the verification is in the document \texttt{verifyLemmas.sage} of \cite{Code}.

As in the rank 3 case, if $I \subseteq [n]-[4]$ is an independent set, the structure of $h(\Gamma_{I})$ depends heavily on the structure of $\{h(\Gamma_{I'})\ |\ I' \subseteq I\}$.  The following lemmas illustrate the structural relationships that will be essential to our algorithm. 

\begin{lem} \label{rank4-structure-edges} Assume that $I = \{x,y\}$. Then the following statements hold: 
\begin{enumerate}
\item If $h(\Gamma_I) = (1,1,0)$ and $|\B_x| =1$, then $|\B_y|\ge 2$: 
\item If $h(\Gamma_I) = (1,1,1)$ then one of the following is true: 
\begin{enumerate}
\item $\min\{|\B_x|, \,|\B_y|\} = 1$ and $\max\{|\B_x|, \,|\B_y|\}\ge3$ , or 
\item $\min\{|\B_x| ,\, |\B_y|\} \ge 3$. 
\end{enumerate}
\item If $h(\Gamma_I) = (1,2,0)$ then $\min\{|\B_x| ,\, |\B_y|\} \ge 2$. 
\item If $h(\Gamma_I) = (1,2,1)$ then $|\B_x|\in \{2,4\}$ and $|\B_y|\in\{2,4\}$. 
\item If $h(\Gamma_I) = (1,2,2)$ then $\min\{|\B_x|, \,|\B_y|\} \ge 2$ and $\max\{|\B_x|, \,|\B_y|\}\ge3$. 
\item If $h(\Gamma_I) = (1,2,3)$ then $\min\{|\B_x|,\, |\B_y|\} \ge 3$. 
\end{enumerate}
\end{lem}

\begin{lem}\label{rank-4-structure-faces} Assume that $I=\{x,y,z\}$. Then the following statements are true: 
\begin{enumerate}
\item If $h(\Gamma_I) = (1,1)$, then $\max\{|\B_x|, \,|\B_y|,\, |\B_z|\} \ge 2$.
\item If $h(\Gamma_I) = (1,2)$, then at most one of $|\B_x|$, $|\B_y|$, $|\B_z|$ is equal to one.
\item If $h(\Gamma_I) =(1,3)$, then $\min\{|\B_x|, \,|\B_y|,\, |\B_z|\} \ge 2$.
\end{enumerate}
\end{lem}

These lemmas serve as the motivation for the design of Algorithm \ref{algor}.  We believe that the main reason that the algorithm produces a pure $O$-sequence is hidden behind these inequalities.  It seems to us that the key to proving Stanley's conjecture in higher rank is to understand how these structural inequalities generalize in higher dimensions.   A better understanding and a generalization of these inequalities to higher dimensions would hopefully lead us to a full solution of the conjecture. 


\subsection{The algorithm}
We now present the algorithm that will give the main result of the paper. The heuristic motivation for this algorithm is that we choose the lexicographically smallest possible set of monomials at each step subject to the requirements set forth by Conjecture \ref{ourconjecture}.  The lemmas from the previous section are key to understanding why we handle the cases the way we do. To simplify notation we write $\Gamma_{x,y,z} := \Gamma_{\{x,y,z\}}$, $\Gamma_{x,y} := \Gamma_{\{x,y\}}$ and $\Gamma_{x} := \Gamma_{\{x\}}$.
\begin{algo}\label{algor}{\bf Constructing a pure rank 4 multicomplex}  \\ 
\boxed{{\bf INPUT:}\text{ A rank 4 matroid $\Delta$ whose vertex set is ordered}}\\
\boxed{{\bf OUTPUT: } \text{A pure multicomplex }\OO \text{ whose }F\text{-vector is } h(\Delta).}\\
{\bf STEP 0:} Reorder the vertices of the matroid as follows. Pick the lexicographic smallest basis of $\Delta$ and relabel its elements as $\{1,2,3,4\}$. For the remaining vertices declare $x<y$ if $|\B_x|<|\B_y|$ and keep the original relative order of the vertices with $|\B_x|$ fixed. Relabel this ordered set as $[n]-\{1,2,3,4\}$ preserving the new order. \\
{\bf STEP 1:} Construct the family of independent sets $I$ with $I\cap \{1,2,3,4\} = \emptyset$ and separate them into five groups $A_0,\, A_1, \, A_2, \, A_3, \,A_4$, where $A_i$ contains all the elements of size $i$. Initialize the list of monomials $\OO$ to the empty list.\\
{\bf STEP 2:} Add the monomial 1 to the list. It corresponds to the empty set, the only element of $A_0$. \\
{\bf STEP 3:} For each $I=\{x\} \in A_1$ add the monomial $x^t$, for $1\le t \le |\B_I|$, to $\OO$. \\
{\bf STEP 4:} For each $I=\{x,y\} \in A_2$ with $x<y$,  we split in cases according to the $h$-vectors of $\Gamma_{x,y}, \Gamma_x, \Gamma_y$ using the following rules: 
\begin{enumerate}
\item If $h(\Gamma_{x,y}) = (1,0,0)$, then add the monomial $xy$ to $\OO$. 
\item Else if $h(\Gamma_{x,y}) = (1,1,0)$, then there are two subcases to consider: 
\begin{enumerate}
\item[i.] If $|\B_x| = 1$ add the monomials $xy$ and $xy^2$ to $\OO$. 
\item[ii.] Else add the monomials $xy$ and $x^2y$ to $\OO$. 
\end{enumerate}
\item Else if $h(\Gamma_{x,y}) = (1,1,1)$, then there are two subcases to consider: 
\begin{enumerate}
\item[i.]  If $|\B_x| = 1$ add the monomials $xy,\, xy^2$ and $xy^3$ to $\OO$;
\item[ii.] Else add the monomials $xy, \, x^2y$ and $x^3y$ to $\OO$. 
\end{enumerate}
\item Else if $h(\Gamma_{x,y}) = (1,2,0)$, add the monomials $xy,\, x^2y$ and $xy^2$ to $\OO$. 
\item Else if $h(\Gamma_{x,y}) = (1,2,1)$, add the monomials $xy,\,x^2y, xy^2, x^2y^2$ to $\OO$.  
\item Else if $h(\Gamma_{x,y}) = (1,2,2)$ then there are two subcases to consider: 
\begin{enumerate}
\item[i.] If $|\B_x| < 3$ add the monomials $xy,\, x^2y,\, xy^2,\, x^2y^2$ and $xy^3$ to $\OO$. 
\item[ii.] Else add the monomials $xy\, x^2y, xy^2, x^3y$ and $xy^3$ to $\OO$. 
\end{enumerate}
\item Else, add the monomials $xy, x^2y, xy^2, x^3y, x^2y^2$ and  $xy^3$ to $\OO$. 
\end{enumerate}
{\bf STEP 5:} For each $I = \{x,y,z\} \in A_3$ with $x<y<z$, we split in cases according to the values of $h(\Gamma_{x,y,z}), \, h(\Gamma_{x,y}),\, h(\Gamma_{x,z})$ and $h(\Gamma_{y,z}), \, 
h(\Gamma_x), \, h(\Gamma_y), \, h(\Gamma_z)$.
\begin{enumerate}
\item If $h(\Gamma_{x,y,z}) = (1,0)$, then add the monomial $xyz$ to $\OO$.
\item Else if $h(\Gamma_{x,y,z})= (1,1)$ there are several subcases:
\begin{enumerate}
\item[i.] If $h(\Gamma_{x,y}) = (1,0,0)$ add the monomials $xyz$ and $xyz^2$ to $\OO$. 
\item[ii.] Else if $h(\Gamma_{x,z}) = (1,0,0)$, then add the monomials $xyz$ and $xy^2z$ to $\OO$. 
\item[iii.] Else if $h(\Gamma_{y,z}) = (1,0,0)$, then add the monomials $xyz$ and $x^2yz$ to $\OO$.  
\item[iv.] Else if $h(\Gamma_{x,y}) = (1,1,0)$ or $h(\Gamma_{x,y})= (1,1,1)$ we split in two cases: 
\begin{enumerate}
\item[a.] if $|\B_x|=1$ add the monomials $xyz$ and $xy^2z$ to $\OO$;
\item[b.] else add the monomials $xyz$ and $x^2yz$ to $\OO$.
\end{enumerate}
\item[v. ] Else if $h(\Gamma_{x,z}) = (1,1,0)$, then 
\begin{enumerate}
\item[a.] if $|\B_x|=1$ add the monomials $xyz$ and $xyz^2$ to $\OO$;
\item[b.] else add monomials $xyz$ and $x^2yz$. 
\end{enumerate}
\item[vi.] Else if $h(\Gamma_{y,z}) = (1,1,0)$ or $h(\Gamma_{y,z}) = (1,1,1)$, then
\begin{enumerate}
\item[a.] if $|\B_y|=1$ add the monomials $xyz$ and $xyz^2$;
\item[b.] else add monomials $xyz$ and $xy^2z$. 
\end{enumerate}
\item[vii.] Else add the monomials $xyz$ and $x^2yz$. 
\end{enumerate}
\item Else if $h(\Gamma_{x,y,z}) = (1,3)$, we again have several cases: 
\begin{enumerate} 
\item[i.] If $|\B_x| = 1$ add $xyz, \, xy^2z$ and $xyz^2$ to $\OO$. 
\item[ii.] Else if $|\B_y| = 1$ add $xyz, x^2yz$ and $xyz^2$ to $\OO$.
\item[iii.] Else if $|\B_z|=1$ add $xyz, x^2yz$ and $xy^2z$ to $\OO$. 
\item[iv.] Else if $h(\Gamma_{xy})=(1,1,0)$ or $h(\Gamma_{x,y})= (1,1,1)$, then 
\begin{enumerate}
\item[a.] if $|\B_x|>|\B_y|$ add the monomials $xyz,\, xy^2z$ and $xyz^2$ to $\OO$;
\item[b.] else add $xyz, \, x^2yz$ and $xyz^2$ to $\OO$.
\end{enumerate}
\item[v.] Else if $h(\Gamma_{xz})=(1,1,0)$ or $h(\Gamma_{x,z}) = (1,1,1)$, then 
\begin{enumerate}
\item[a.] if $|\B_x|>|\B_z|$ add the monomials $xyz,\, xy^2z$ and $xyz^2$ to $\OO$;
\item[b.] else add $xyz, \, x^2yz$ and $xy^2z$ to $\OO$.
\end{enumerate}
\item[vi.]Else if $h(\Gamma_{yz})=(1,1,0)$ or $h(\Gamma_{y,z})= (1,1,1)$, then 
\begin{enumerate}
\item[a.] if $|\B_y|>|\B_z|$ add the monomials $xyz,\, x^2yz$ and $xyz^2$ to $\OO$;
\item[b.] else add $xyz, \, x^2yz,\, xy^2z$ to $\OO$
\end{enumerate}
\item[vii.] Else if $h(\Gamma_{x,y})=(1,2,0)$, add the monomials $xyz,\, x^2yz$ and $xy^2z$ to $\OO$.
\item[viii.] Else if $h(\Gamma_{x,z})=(1,2,0)$, add the monomials $xyz,\, x^2yz$ and $xyz^2$ to $\OO$. 
\item[ix.] Else if $h(\Gamma_{y,z})=(1,2,0)$, add the monomials $xyz,\, xy^2z$ and $xyz^2$ to $\OO$.
\item[x.] Else add the monomials $xyz, \, x^2yz$ and $xy^2z$ to $\OO$. 
\end{enumerate}
\item Else if $h(\Gamma_{xyz})=(1,3)$, then add the monomials $xyz,\, x^2yz,\, xy^2z$ and $xyz^2$ to $\OO$.
\end{enumerate}
{\bf STEP 6:} For each $I = \{w,x,y,z\} \in A_4$ add the monomial $wxyz$ to $\OO$. \\
\end{algo}

For a based matroid $(\Delta, B,<)$, let $\FF(\Delta, B,<)$ be the output of the algorithm when we input $\Delta$. We want to show that $\FF(\Delta, B,<)$ is indeed a pure multicomplex. It is sufficient to show that claim holds as $(\Delta,B)$ ranges over the elements of ${\cal U}_4$. To do this we implemented the algorithm in Sage and verified computationally that it produces pure $O$-sequences for each element of ${\cal U}_4$ (up to ordered isomorphism). Matroids with $8$ elements are classified in \cite{Database}, \cite{Matroid-Database} and \cite{Mayhew-Royle}. We use the classification to produce all isomorphism classes of based matroids and get the desired multicomplex for small matroids. 

\subsection{Matroids with 8 elements}
We begin by establishing a lemma that is essential to simplify the computations that we have to do to prove that the algorithm works in general. 
\begin{lem}\label{h-vecs} Let $\Delta$ and $\Delta'$ be rank $4$ matroids whose ground set is $[k]$ for $5\le k \le 8$. Assume the following assertions hold:
\begin{enumerate}
\item $\{1,2,3,4\}$ is a basis of both.
\item For all $A \subseteq [k]-\{1,2,3.4\}$, $A$ is independent in $\Delta$ if and only if $A$ is independent in $\Delta'$. 
\item For $I \subseteq [k]-\{1,2,3,4\}$ independent in $\Delta$, $h(\Gamma_I(\Delta)) = h(\Gamma_I(\Delta'))$. 
\end{enumerate}
Then ${\cal F}(\Delta,\{1,2,3,4\},<) = {\cal F}(\Delta',\{1,2,3,4\},<)$. \end{lem}
\begin{proof} Since $|B_x(\Delta)| = \sum_{i=0}^{3} h_i(\Gamma_{\{x\}}(\Delta))$ all the cases considered in every step of the algorithm only depend on $h(\Gamma_I(\Delta))$ as $I$ ranges among the independent sets of $\Delta$ that do not intersect $\{1,2,3,4\}$. This proves the claim. 
\end{proof}

Lemma \ref{h-vecs} together with the fact that every based matroid $(\Delta, B,<)\in {\cal U}_4$ is isomorphic to a based matroid $(\Delta', \{1,2,3,4\},<)$ where $E(\Delta') = \{1,2,3,4,5,6,7,8\}$ and $<$ is the natural order imply that it is enough to follow the following procedure.  We outline this computational procedure further in Section \ref{code-section}. 
\begin{itemize}
\item Start with an empty list of \emph{matroids-to-check} that saves the families of $h$-vectors of the form $(h(\Gamma_I) \, | \, I \cap \{1,2,3,4\}=\emptyset)$. For each such family of $h$-vectors we will store exactly one matroid with those restricted $h$-vectors. 
\item For each based matroid $(\Delta, \{1,2,3,4\})$ with vertex set $\{1,2,3,4,5,6,7,8\}$ determine if there is a matroid $(\Delta', \{1,2,3,4\})$ whose family of $h$-vectors $\{h(\Gamma'_I), \, |\, I\cap \{1,2,3,4\} = \emptyset\} = \{h(\Gamma_I), \, | \, I\cap\{1,2,3,4\}\}$. If no such matroid exists, then add $\Delta$ and its list of restricted $h$-vectors to the list of \textit{matroids-to-check}. 
\item Run Algorithm \ref{algor} on every matroid in the list of \emph{matroids-to-check}. 
\item Verify that the output of the algorithm for every complex to check is indeed a pure order ideal. 
\end{itemize}

The result of the test is positive. The total number of matroids in the list of matroids to verify is $9085$. Below is and example of the outputs of the algorithm. 

\begin{ex}Let $\Delta$ be the dual matroid of the Fano matroid from Example \ref{fano}, that is, the matroid whose bases are the complements of the bases of the Fano matroid with the labels as in the previous example. The restricted $h$-vectors and corresponding monomials in this case are shown in the following table.  Once again, we can easily check that the resulting family of monomials is a pure multicomplex.
\begin{center}
\begin{tabular}{|c|c|c|} \hline
$I$ & $h(\Gamma_I)$ & Monomials \\ \hline
$\{5\}$ & $(1,1,1,0)$ & $x_5,\ x_5^2,\ x_5^3$\\ \hline
$\{6\}$ & $(1,1,1,0)$ & $x_6,\ x_6^2,\ x_6^3$\\ \hline
$\{7\}$ & $(1,1,1,0)$ & $x_7,\ x_7^2,\ x_7^3$\\ \hline
$\{{5,6}\}$ & $(1,2,2)$ & $x_5x_6,\ x_5^2x_6,\ x_5x_6^2,\ x_5^3x_6,\ x_5x_6^3$ \\ \hline
$\{{5,7}\}$ & $(1,2,2)$ &  $x_5x_7,\ x_5^2x_7,\ x_5x_7^2,\ x_5^3x_7,\ x_5x_7^3$\\ \hline
$\{{6,7}\}$ & $(1,2,2)$ &  $x_6x_7,\ x_6^2x_7,\ x_6x_7^2,\ x_6^3x_7,\ x_6x_7^3$\\ \hline
$\{{5,6,7}\}$ & $(1,2)$ & $x_5x_6x_7,\ x_5^2x_6x_7,\ x_5x_6^2x_7$ \\ \hline
\end{tabular}
\end{center}

\end{ex}
\section{Summary of code} \label{code-section}

This section contains a summary of the computational approach undertaken to verify Conjecture \ref{ourconjecture} in ranks three and four.  The code and data are available online at \cite{Code}.

In order to computationally verify our results, we used the database of matroids in \cite{Database}.  This gave us a representative of each isomorphism class of matroids of rank three (respectively four) on at most six (respectively eight) elements.  The files \texttt{matroidsnXrY.sage} contain the bases of each isomorphism class of matroids of rank $Y$ on $X$ elements.  Initially, the bases are stored as $0/1$ lists that encode the $Y$-element subsets of $[X]$ under the reverse lexicographic order.  The code in the file \texttt{constructMatroids.sage} converts each $0/1$ list into a list of facets/bases of a simplicial complex.

For each such isomorphism class, any potential reordering of the ground set would give a different based matroid with a different initial basis under the lexicographic shelling order.  For each based matroid, we compute the set of restricted $h$-vectors $\{h(\Gamma_I)\ |\ I \subseteq [n]\setminus[d]\}$.  This analysis is also done in the file \texttt{constructMatroids.sage}. For each unique set of restricted $h$-vectors, we stored the corresponding ordered matroid and list of restricted $h$-vectors for further analysis.  The restricted $h$-vectors and corresponding matroids are stored in two separate lists in the files \texttt{hvecsnXrY.sage} for further analysis.  From the original list of 1331 isomorphism classes of matroids of rank four on at most eight elements, we constructed a list of 9085 based matroids to be examined.

Now that we have saved a permanent record of all possible sets of restricted $h$-vectors $\{h(\Gamma_I)\ |\ I \subseteq [n]\setminus[d]\}$ for all matroids of rank three (or four) on at most six (or eight) elements, we are able to computationally verify Lemmas \ref{rank3-link-structure}, \ref{rank4-structure-edges}, and \ref{rank-4-structure-faces} and implement our Algorithms \ref{construct-multicomplex} and \ref{algor}.  The Lemmas are verified using the code in \texttt{verifyLemmas.sage}.  The algorithms are implemented in the files \texttt{constructMonomialsRank3.sage} and \texttt{constructMonomialsRank4.sage}.  Finally, we wrote code in the file \texttt{verifyPureOSequence.sage} to test whether a given list of monomials is a pure multicomplex.  We ran these tests in the file \texttt{verifyAlgorithm.sage}, and the output verified that the family of monomials constructed by our algorithm was indeed a pure order ideal in each case.  Finally we verified that the $F$-vector of the output of the algorithm coincides with the $h$-vector of the input. The $h$-vector of the input is computed using the standard \texttt{SimplicialComplex} class of Sage.


\bibliography{biblio}
\bibliographystyle{plain}

\end{document}